\documentclass[12pt]{article}
\usepackage[utf8]{inputenc}

\usepackage{amsfonts}
\usepackage{amsthm}
\usepackage{amssymb}
\usepackage{amsmath}
\usepackage{amscd}

\usepackage[english]{babel}
\usepackage{indentfirst}
\usepackage{graphicx}

\usepackage{mathtools}
\numberwithin{equation}{section}

\usepackage{enumitem}
\usepackage{varwidth}

\theoremstyle{plain}
\newtheorem{Th}{Theorem}[section]

\newtheorem{Prop}[Th]{Proposition}

\theoremstyle{definition}
\newtheorem{Def}[Th]{Definition}

\newtheorem{Rem}[Th]{Remark}
\newtheorem{?}[Th]{Problem}

\newtheorem*{Claim}{Claim}
\newtheorem*{Fact}{Fact}

\DeclareMathOperator{\Hom}{Hom}
\DeclareMathOperator{\Aut}{Aut}
\DeclareMathOperator{\gal}{Gal}
\DeclareMathOperator{\im}{Im}
\DeclareMathOperator{\Ker}{Ker}
\DeclareMathOperator{\charar}{char}

\usepackage{scalerel,stackengine}
\stackMath
\newcommand\reallywidehat[1]{%
	\savestack{\tmpbox}{\stretchto{%
			\scaleto{%
				\scalerel*[\widthof{\ensuremath{#1}}]{\kern-.6pt\bigwedge\kern-.6pt}%
				{\rule[-\textheight/2]{1ex}{\textheight}}
			}{\textheight}%
		}{0.5ex}}%
	\stackon[1pt]{#1}{\tmpbox}%
}
\title{On the First Cohomology of Local Units}
\author{Wei Yin}
\date{}
\AtEndDocument{\bigskip{\footnotesize
		\addvspace{\medskipamount}
		\noindent
		\textsc{Wei Yin. University of California, San Diego.}
		
		\textit{E-mail address}: \texttt{wey101@ucsd.edu}
}}
\begin{document}

	\maketitle
	\begin{abstract}
		The groups of units $U^i_L$ of a local field $L$ play an important role in algebraic number theory, especially in class field theoretic topics. Therefore, it is interesting to study these groups from a cohomological point of view. In this article, we study and compute the first cohomology of $U_L^1$, $U^2_L$ and $U^3_L$ under certain mild hypotheses, and discuss some results about general $U^i_L$'s.
	\end{abstract}
	
	\tableofcontents
	\subsubsection*{Acknowledgements}
	\thispagestyle{empty}
	
	The author is grateful to Prof. Dr. Nickel for his inspiring discussions and warm encouragement during the pandemic.
	
	\clearpage
	\section{Notations and basic results}
		\subsection{Local fields}
			In this section, we summarize the notations and results that will be used in this article. Serre's book \cite{locf} is our main source, where all the results in this section can be found.
	
			Throughout this article, $K$ denotes a complete field with a discrete valuation $\nu_{K}$. $K$ is endowed with the topology defined by $\nu_{K}$. The corresponding valuation ring is denoted by $\mathcal{O}_{K}\coloneqq\{x\in K| \nu_{K}(x)\geqslant0\}$, and its maximal ideal by $\mathfrak{m}_{K}$. Its residue field is $\kappa=\mathcal{O}_{K}/\mathfrak{m}_{K}$, and $U_{K}=\mathcal{O}_{K}-\mathfrak{m}_{K}$ is the group of units. We assume that the residue field is finite with positive characteristic $p>0$. 
	
			Let $L$ be a finite Galois extension of $K$ with $\gal(L/K)=G$, and let $\mathcal{O}_{L}$ be the integral closure over $\mathcal{O}_{K}$ in L. We know that $\mathcal{O}_{L}$ is also a complete discrete valuation ring, with a valuation $\nu_{L},$ its maximal ideal $ \mathfrak{m}_{L},$, the residue field $\lambda$ and the group of units $U_{L}$. With the assumption above, the residue field extension is an extension of finite fields. The ramification index is denoted by $e$ and the degree of the residue field extension is denoted by $f$. We have $ef=[L:K]$. Here, we will always make the convention that $\nu_{L}(L^{\times})=\mathbb{Z}$ and $\nu_{L}(K^{\times})=e\mathbb{Z}$. We fix a choice of uniformizer of $L$ and denote it by $\pi_L$, so $\nu_{L}(\pi_L)=1$. 
			\subsection{Unit groups and ramification groups}
			For $i\geqslant0$, we denote by $U^{i}_{L}$ the elements $x\in\mathcal{O}_{L}$ such that $x\equiv1\mod\pi_L^{i}$. In other words,
			 $U_{L}^{i}=1+\mathfrak{m}_{L}^{i}$. By convention, we set $U^{0}_{L}=U_{L}$. These groups $U_{L}^{i}$ give a decreasing sequence of closed subgroups of $U_L$, and they form a neighborhood base of $1$ for the topology induced on $U_L$ by $L^{\ast}$. Further, we have the following result:
			
			 $$U_{L}\cong\varprojlim\limits_{i\geqslant 0} U_{L}/U_{L}^{i}.$$
	 		\begin{Rem}
	 			In fact we have for any $k$, 
	 			$$U_{L}^{k}\cong\varprojlim\limits_{j\geqslant k}U_{L}^{k}/U_{L}^{j}$$
	 		\end{Rem}
			There is a filtration for the Galois group $G$. Let $i\geqslant-1$ be an integer. Define
				$$ G_i\coloneqq\{\sigma\in G|\  \nu_{L}\big(\sigma(a)-a\big)\geqslant i+1 \ \text{for any}\  a\in\mathcal{O}_{L}\}.$$
	
			These groups are called ramification groups and $G_i$ is called the $i$-th ramification group of $G$. We have the following well-known result:
	
			\begin{Prop}
			The $G_i$ form a decreasing sequence of normal subgroups of $G$. $G_{-1}$ is the full Galois group $G$, and $G_{0}$ is the inertia subgroup of $G$. Moreover, $G/G_0$ is canonically isomorphic to $\gal(\lambda/\kappa)$.
			\end{Prop}
			
			Next, we consider the relationship between unit groups and ramification groups.
		 	We first recall that there is a canonical isomorphism $$\rho_0:U_L/U^1_L\xrightarrow{\sim}\lambda^{\times},$$ 
		 	given by 
		 	$$\rho_0(u)=\widehat{u},$$
		 	where $\widehat{u}$ is the image of $u$ in the residue field $\lambda$.
		 	Also, for $i\geqslant1,$ there are  non-canonical isomorphisms  $$\rho_i:U^i_L/U^{i+1}_L\xrightarrow{\sim}\lambda$$ 
		 	such that if $x=1+u\cdot\pi_L^i$, then
		 	$$\rho_i(x\cdot U_L^{i+1})=\widehat{u}.$$
		 	When $i\geqslant1$, these $\rho_i$'s are not canonical because their definition involves a choice of the uniformizer $\pi_L$. What is worse, $\rho_i$ is only an isomorphism of abelian groups, not an isomorphism of $G$-modules. We will investigate this issue more closely in Section 4.
		 	
			Next, we recall the following two propositions cited from Serre's book \cite{locf}.
			
			\begin{Prop}\label{serre1}
				Let $i$ be a non-negative integer and $\sigma$ be an element in $G_0$. In order for $\sigma$ to belong to $G_i$, it is sufficient and necessary that
					$$\dfrac{\sigma(\pi_L)}{\pi_L}\equiv1\mod \mathfrak{m}_{L}^{i}.$$
				 In other words, $\sigma\in G_i$ if and only if $\dfrac{\sigma(\pi_L)}{\pi_L}\in U_{L}^{i}$. 
			\end{Prop}
		
			Since for any $\sigma\in G$, the fraction $\dfrac{\sigma(\pi_L)}{\pi_L}\in U_L$ is a unit, and we may consider the map $f_{\pi}:G\to U_L$ defined by $$f_{\pi}(\sigma)=\dfrac{\sigma(\pi_L)}{\pi_L}.$$ We know from Proposition \ref{serre1} that $f_{\pi}(G_{i})\subseteq U_{L}^{i}$.
			For each index $i$, we may consider the restriction of $f_\pi$ on $G_i$, and compose it with the quotient map 
			$$G_i\xrightarrow{f_\pi} U_{L}\xrightarrow{} U_{L}^{i}/U_{L}^{i+1}.$$
			We shall denote by $\theta_{i}$ this composed map $$\theta_{i}: G_i\to U_{L}^{i}/U_{L}^{i+1}.$$ 
			Due to the above proposition, $\theta_{i}$ factors through $G_{i+1}$. By abuse of notations, we will also denote by $\theta_i$ the induced map:
			$$\theta_i: G_i/G_{i+1}\rightarrow U_{L}^{i}/U_{L}^{i+1}.$$  
			
			We rephrase 
			\begin{Prop}\label{serreprop}
				Let $i$ be a non-negative integer. The map which assigns $\dfrac{\sigma(\pi)}{\pi}$ to $\sigma\in G_i$, induces by passage to the quotient an isomorphism $\theta_{i}$ of the quotient group $G_i/G_{i+1}$ onto a subgroup of the group $U_{L}^{i}/U_{L}^{i+1}$. The isomorphism is independent of the choice of the uniformizer. In particular, the map $\theta_{i}: G_i/G_{i+1}\to U_{L}^{i}/U_{L}^{i+1}$ is injective.
			\end{Prop}

		\subsection{Invariants of the unit groups}\label{section-invariants}
		In this subsection, we consider the $G$-invariants $(U^i_L)^G$ of the groups of units $U^i_L$. 
		
		We first show that $(U_L)^G = U_K$. Let $s\in U_L$ be such that $\sigma(s)=s$ for any $\sigma\in G$. This implies that $s\in K$. From the equality $s\cdot t=1$, we draw $\sigma(t)=t$, for any $\sigma$. Therefore the inverse $t$ is also in $K$. Thus $s\in U_K$, and hence $(U_L)^G\subseteq U_K$. The other inclusion is obvious, hence we conclude that $(U_L)^G=U_K.$ 
		
		Next, we show $(U_L^1)^G=U^1_K.$ Let $s\in U^1_L$ be such, that $\sigma(s)=s$ for any $\sigma\in G$. As in the previous paragraph, this implies $s\in U_K$. Since $\nu_{L}(s-1)>0$, it follows that $\nu_K(s-1)>0$. Therefore, $s\in U^1_K.$ So, we have $(U_L^1)^G\subseteq U^1_K.$ On the other hand, one has $U^1_K\subseteq (U^1_L)^G$. Thus, we draw the equality $(U_L^1)^G=U^1_K.$
		
		However, it is not true in general that $(U_L^i)^G=U^i_K$, as the following argument shows. In fact, we have the following:
		\begin{Prop}For any $i\geqslant0$, one has:
			$$(U^i_L)^G=U_K^j,$$ 
			where $j=\lceil \dfrac{i}{e} \rceil.$
			In particular, if $L/K$ is unramified, then $(U^i_L)^G=U_K^i$ for any $i\geqslant0$.
		\end{Prop}
		\begin{proof}
			Set $i'=j\cdot e$. 
			Then, one has the obvious inequality $i'\geqslant i$. Hence, the inclusion $U^{i'}_L\subseteq U^i_L$ yields the inclusion of invariants 
			$$(U^{i'}_L)^G\subseteq (U^i_L)^G.$$
			
			On the other hand, assume $s\in(U_L^i)^G$, then, as before, we have $s\in U_K^1$. Consider $v=\nu_{L}(s-1)$ and $v'=\nu_{K}(s-1), v, v'\in\mathbb{Z}.$ Since $s\in U^i_L$, one has $v\geqslant i$.
			
			Note that one has the identity
			$$\nu_{K}(x)=\nu_{L}(x)/e,$$
			for any $x\in K$.
			This shows that $v'=v/e$. Since $v'\in\mathbb{Z},$ it follows that $e|v$. Combining this with the inequality, one has $$\dfrac{v}{e}\geqslant\dfrac{i}{e},$$
			and since $\dfrac{v}{e}$ is an integer, one has 
			$$\dfrac{v}{e}\geqslant\lceil\dfrac{i}{e}\rceil=j.$$
			Namely, $v\geqslant i'=ej$, and $s\in U^{i'}_L$. Therefore,
			$s\in (U^{i'}_L)^G$. This shows the other inclusion and we draw:
			$$(U^{i}_L)^G=(U^{i'}_L)^G.$$
			
			It remains to prove that $(U^{i'}_L)^G=U^j_K$. Indeed, suppose $s\in U^j_K$. Then we have $\nu_{K}(s-1)\geqslant j$, hence $\nu_{L}(s-1)\geqslant j\cdot e=i$, meaning $s\in U^{i'}_L$. This shows $U^j_K\subseteq(U^{i'}_L)^G.$ On the other hand, suppose $s\in(U^{i'}_L)^G$. Then $\nu_{L}(s-1)\geqslant i'=ej,$ hence $\nu_{K}(s-1)\geqslant j$, meaning $s\in U^j_K.$ Thus, $(U^{i'}_L)^G\subseteq U^j_K$. This settles the proof.
		\end{proof}
		
		As a final remark, we cite the following result from \cite{cohomology}.
		\begin{Prop}
			If the local field extension $L/K$ is unramified, then the group of units $U_L$ and the group of principal units $U^1_L$ are cohomologically trivial.
		\end{Prop}
		In fact, the proof therein implies that all the higher unit groups $U^i_L$ are cohomologically trivial if $L/K$ is unramified. 
		In this paper, we will study the behavior of the cohomology groups assuming that the local field extension $L/K$ is ramified, and it turns out that the sizes of the cohomology groups are related to the ramification indices.

	\section{First cohomology of the group of units $U_{L}$}
	It is known in \cite{nickel} and \cite{cohloc} that the first cohomology group $H^1(G,U_L)$ of $U_L$ is cyclic of order $e$. We present the proof, with an extra emphasis on the explicit 1-cocycle $f_\pi$, which generates the cohomology group and plays a fundamental role in later parts of this article.
	\subsection{The cohomology class $f_{\pi}$ }
	Let $f_{\pi} : G\rightarrow U_{L}$ be the map defined in the previous section. One sees easily that the map $f_{\pi}$ defines a 1-cocycle relation so it represents a cohomology class in $H^{1}(G,U_{L})$. Note that it is not a coboundary because $\pi$ itself is not in $U_L$, so the cohomology class is not trivial. Although the definition of $f_{\pi}$ as a 1-cocycle depends on the choice of the uniformizer $\pi$, the cohomology class does not. Indeed, let $\pi^{\prime}$ be another uniformizer of $L$, then $\pi$ and $\pi^{\prime}$ differ by a unit, i.e., $\pi^{\prime}=u\cdot\pi$ for some $u\in U_{L}$. We look at the cocycle defined by $\pi^{\prime}$:
	\begin{align*}
f_{\pi^{\prime}}(\sigma)=\frac{\sigma(\pi^{\prime})}{\pi^{\prime}}&=\frac{\sigma(u\cdot\pi)}{u\cdot\pi}\\
&=\frac{\sigma(u)}{u}\cdot\frac{\sigma(\pi)}{\pi}\\
&=f_{u}(\sigma)\cdot f_{\pi}(\sigma).
	\end{align*}
Here, the map $f_u$ defined by $f_u(\sigma)=\dfrac{\sigma(u)}{u}$ is a 1-coboundary on $U_L$, which gives the trivial cohomology class. Thus, the cohomology classes defined by $f_{\pi^{\prime}}$ and $f_{\pi}$ are the same. Namely, the cohomology class does not depend on the choice of the uniformizer $\pi_L$. By abuse of notation, we shall denote the corresponding cohomology class by $f_{\pi}$ as well, and show that it is a generator of $H^{1}(G, U_L)$. 
	
	\subsection{The cohomology group  $H^{1}(G,U_{L})$}
		\begin{Prop}
		The cohomology group $H^{1}(G,U_{L})$ is a cyclic group whose order is equal to the ramification index $e$ and is generated by the cohomology class $f_\pi$.
		\end{Prop}
	\begin{proof}
		Consider the exact sequence of $G$-modules:
		$$1\to U_{L}\to L^{\times}\xrightarrow{\nu_{L}}\mathbb{Z}\to 0,$$ where $\mathbb{Z}$ is endowed with the trivial $G$-action.
		This sequence induces a long exact sequence of cohomology:
		$$1\to (U_{L})^{G}\to (L^{\times})^{G}\xrightarrow{\nu_{L}}(\mathbb{Z})^{G}\xrightarrow{\delta}H^{1}(G,U_{L})\to H^{1}(G,L^{\times})\to \cdots.$$
		By Hilbert 90, we have $H^{1}(G,L^{\times})=0$. Hence, the first several terms give the following:
		$$1\to U_{K}\to K^{\times}\xrightarrow{\nu_{K}}\mathbb{Z}\xrightarrow{\delta}H^{1}(G,U_{L})\to 0.$$ Thus $$H^{1}(G,U_{L})\cong\mathbb{Z}/\nu_{L}(K^{\times})\cong\mathbb{Z}/e\mathbb{Z}.$$ This shows that $H^{1}(G,U_{L})$ is a cyclic group of order $3$. 
		
		Next, we look closely at the map $\delta$. It is clear that $\delta(1)$ is the generator of $H^{1}(G,U_{L})$. By definition, to obtain $\delta(1)$, one chooses an element $x\in L^{\times}$ such that $\nu_L(x)=1$. The canonical choice is $\pi_L$, then the fraction $\sigma(\pi_L)/\pi_L$ defines the 1-cocycle $\delta(1)$. This is exactly the map $f_\pi$ we defined in the previous section. As a result, we have the following:
		\begin{Fact}
			The cohomology class $\delta(1)$ is represented be the 1-cocycle $f_{\pi}$, which is defined by $$f_{\pi}(\sigma)=\dfrac{\sigma(\pi_L)}{\pi_L}$$ for any $\sigma\in G$.
		\end{Fact}
		Therefore, we conclude that $H^{1}(G,U_{L})$ is a cyclic group of order $e$, with $f_\pi$ as a generator. 
		\end{proof}

	\section{First cohomology of principal units $U_{L}^{1}$ }
	Our goal in this section is to show that $H^{1}(G,U_{L}^{1})$ is a cyclic group of order $w$, with $f_\pi^{t}=(f_\pi)^t$ as a generator. Here, $w$ and $t$ stand for the wild ramification index and tame ramification index respectively. In fact, our claim can be seen from the factorization $$U_{L}\cong U_{L}^1\times \lambda^{\times},$$ 
	and henceforth
	$$H^{1}(G,U_{L})\cong H^{1}(G,U^1_{L})\times H^{1}(G,\lambda^{\times}).$$ It follows that $ H^{1}(G,U^1_{L})$ is the $p$-component of the group $H^{1}(G,U_{L})$ and $H^{1}(G,\lambda^{\times})$ is the prime-to-$p$ component. However, we also present the following longer argument, for the method implemented there will play an important role in the next section. 
	\subsection{The preparatory step}
	 We make the following 
	 \begin{Claim}
	 	$$H^{1}(G,U_{L}^{1})=\operatorname{ker}( H^{1}(G,U_{L})\to H^{1}(G,\lambda^{\times})).$$
	 \end{Claim}
	\begin{proof}
		The exact sequence
		$$1\to  U_{L}^{1}\to  U_{L}\xrightarrow{}\lambda^{\times}\to 1.$$
		induces a long exact sequence of cohomology groups :
		$$1\to ({U_{L}^{1}})^G\to(U_{L})^{G}\xrightarrow{}(\lambda^{\times})^{G}\xrightarrow{\delta}H^{1}(G,U_{L}^{1})\to H^{1}(G,U_{L})\xrightarrow{}H^{1}(G,\lambda^{\times})\to \cdots.$$
		Rewrite it as:
		$$1\to {U_{K}^{1}}\to {U_{K}}\xrightarrow{}{\kappa^{\times}}\xrightarrow{\delta}H^{1}(G,U_{L}^{1})\to H^{1}(G,U_{L})\xrightarrow{}H^{1}(G,\lambda^{\times})\to \cdots.$$
		The map ${U_{K}}\xrightarrow{}{\kappa^{\times}}$ is the canonical projection map, and we know the map is surjective, with kernel $U_{K}^{1}$. Thus, the above long exact sequence slits into two exact sequences:
		$$1\to {U_{K}^{1}}\to {U_{K}}\xrightarrow{}{\kappa^{\times}}\to  1,$$
		and	
		$$0\to H^{1}(G,U_{L}^{1})\to H^{1}(G,U_{L})\xrightarrow{}H^{1}(G,\lambda^{\times})\to \cdots.$$
		In particular, the second sequence implies that the map $H^{1}(G,U_{L}^{1})\to H^{1}(G,U_{L})$ induced by the inclusion is injective. As a result, $H^{1}(G,U_{L}^{1})$ can be canonically identified as a subgroup of $H^{1}(G,U_{L})$. More precisely, it is the kernel of the map $ H^{1}(G,U_{L})\to H^{1}(G,\lambda^{\times})$. Note here that this map is induced by the canonical projection map $U_{L}\to \lambda^{\times}$.  This settles our first step.
	\end{proof}
	
	\subsection{Relating with $T=\Hom(G_{0}/G_1,\lambda^{\times})$}
	In this subsection, we embed $H^{1}(G,\lambda^{\times})$ into some group $T$ such that the composition $H^{1}(G,U_{L})\to T$ is easy to understand. Since $H^{1}(G,\lambda^{\times})\to T$ is injective, the composed map 
	$$H^{1}(G,U_{L})\to H^{1}(G,\lambda^{\times})\to T$$ has the same kernel as $H^{1}(G,U_{L})\to H^{1}(G,\lambda^{\times})$, which is $H^{1}(G,U_{L}^{1})$. We will compute $H^{1}(G,U_{L}^{1})$ using the composition in the next subsection.

	In order to figure out the group $T$, we make the following considerations. For $G=\operatorname{Gal}(L/K)$, the normal subgroup $G_0\subseteq G$ and the $G$-module $\lambda^{\times},$ the inflation-restriction exact sequence for this data gives:
	$$0\to H^{1}(G/G_0,(\lambda^{\times})^{G_0})\xrightarrow{\text{inf}}H^{1}(G,\lambda^{\times})\xrightarrow{\text{res}}H^{1}(G_{0},\lambda^{\times}).$$
	
	Notice that $G_0$ acts trivially on $\lambda^{\times}$, hence we have $(\lambda^{\times})^{G_0}=\lambda^{\times}$, and $H^{1}(G_{0},\lambda^{\times})=\Hom(G_{0},\lambda^{\times})$. As a result, the inflation-restriction sequence reads:
	$$0\to H^{1}(G/G_0,{\lambda^{\times}})\xrightarrow{\text{inf}}H^{1}(G,\lambda^{\times})\xrightarrow{\text{res}}\Hom(G_{0},\lambda^{\times})$$
	
	We also notice that $G/G_0\cong\operatorname{Gal}(\lambda/\kappa)$. Hence, by Hilbert 90, the first term $H^{1}(G/G_0,{\lambda^{\times}})$ is 0. For the third term, note that $\lambda^{\times}$ has cardinality coprime to $p$, while $G_1$ is a $p$-group, so any homomorphism from $G_0$ to $\lambda^{\times}$ factors through $G_0/G_1$. Namely, there is a natural isomorphism $$\Hom(G_{0},\lambda^{\times})\cong\Hom(G_{0}/G_1,\lambda^{\times}).$$ 
	
	As a result, the above sequence becomes:
	$$0\to H^{1}(G,\lambda^{\times})\xrightarrow{\text{res}}\Hom(G_{0}/G_1,\lambda^{\times}).$$
	This means that the restriction map is injective. We set $$T\coloneqq\Hom(G_{0}/G_1,\lambda^{\times}).$$ This settles our second step.
	
	\subsection{Computing the kernel of $H^{1}(G,U_{L})\to T$}
	To compute the kernel, we take the generator $f_\pi$ of the group $H^{1}(G,U_{L})$, and consider its image in $T=\Hom(G_{0}/G_1,\lambda^{\times})$. We check that the image of $f_\pi$ is $\theta_{0}$. Indeed, the first map $H^{1}(G,U_{L})\to H^{1}(G,\lambda^{\times})$ is induced from the natural projection $U_{L}\to \lambda^{\times}$, so the image of $f_\pi$ in $H^{1}(G,U_{L})$ is the cohomology class represented by the assignment $$\sigma\mapsto\reallywidehat{\dfrac{\sigma(\pi_L)}{\pi_L}}\in\lambda^{\times}.$$
	Then, the restriction map $H^{1}(G,U_{L})\xrightarrow{\text{res}}\Hom(G_{0},\lambda^{\times})$ merely narrows the domain of $f_\pi$ to the subgroup $G_0$, we see this assignment coincides with the map $\theta_{0}$ defined in the first section, under the canonical identification $U_L/U_L^1=\lambda^{\times}$.
	
	Notice that $\theta_{0}$ is an injective homomorphism from the cyclic group $G_0/G_1$ to $\lambda^{\times}$, so the order of $\theta_{0}$, as an element in the group $T=\Hom(G_{0}/G_1,\lambda^{\times})$, must $|G_0/G_1|=t$, the tame ramification index. Hence the image of $f_\pi$ in the target $T$ is of order $t$. By an easy group theory exercise, we conclude that  the kernel of $H^{1}(G,U_{L})\to T$, which is $H^{1}(G,U_{L}^{1})$, is generated by $f_\pi^{t}$. In particular, its order is $e/t=w$, the wild ramification index. This finishes our calculation.

	We conclude that:
	\begin{Prop}
		The cohomology group ${H^{1}}(G,U_{L}^{1})$ is naturally isomorphic to the subgroup of $H^{1}(G,U_{L})$ generated by $(f_\pi)^{t}$,  where $t=\dfrac{|G_0|}{|G_1|}$ is the tame ramification index. In particular, the order of ${H^{1}}(G,U_{L}^{1})$ is equal to $w=|G_1|$, the wild ramification index.
	\end{Prop}

	\section{The Galois module $\lambda_i$}
	
	\subsection{Group action twisted by 1-cocycles}
	Let $G$ be a finite group and $k$ be a field. Suppose that there is an action of $G$ on $k$ denoted by $\mu:G\times k\to k$, such that $\mu$ is induce from a group homomorphism $G\to\Aut(k)$. We denote this action by $g\ast x$ for $g\in G$ and $x\in k$. This action gives rise to an action of $G$ on $k^\times$, and by abuse of notation we shall also denote it by $\ast$. 
	
	Regarding the action of $G$ on $k^\times$, let $f:G\to k^{\times}$ be a $1$-cocycle. Then, we make the following:
	\begin{Def}
		For the data above, the twisted action $\mu_f:G\times k\to k$ of $\mu$ by $f$ is given by the following formula:
		$$\mu_f(g,a)\coloneqq f(g)\cdot\mu(g,a),$$
		or 
		$$g \ast_f a\coloneqq f(g)\cdot(g\ast a),$$
		where $\cdot$ denotes the multiplication of the field $k$. 
	\end{Def}

	\begin{Prop}
		$\mu_f : G\times k\to k$ defined above is an action of $G$ on $k$.
	\end{Prop}
	\begin{proof}		
		To show that $\mu_f$ is indeed a group action, we need to prove that 
		\begin{itemize}
			\item $e\ast_f a=a$;
			\item $(gh)\ast_f a=g\ast_f(h\ast_f a)$.
		\end{itemize} 
		
		To prove the first identity, we notice that if
		$f:G\to k^{\times}$ is a 1-cocycle, then the equality $f(e)=1$ holds. Indeed, for any $g\in G$, one has \begin{align*}
			f(g)&=f(eg)\\
				  &=f(e)\cdot(e\ast f(g))\\
				  &=f(e)\cdot f(g).
		\end{align*}
		Since $k$ is a field and $f(g)\in k^{\times}$, we can cancel $f(g)$ from the equality $f(g)=f(e)\cdot f(g) $ to get $1=f(e)$, as desired.
		
		Then, we have
		\begin{align*}
			e\ast_f a&=f(e)\cdot( e\ast a)\\
							&=1\cdot (e\ast a)\\
							&=1\cdot a=a.
		\end{align*}
		This settles the first identity.
		
		To prove the second identity $$(gh)\ast_f a=g\ast_f(h\ast_f a),$$
		we first compute the left-hand side. 
		By definition,
		\begin{align*}
			(gh)\ast_f a=&f(gh)\cdot (gh\ast a)
		\end{align*}
		On the other hand, the right-hand side is 
		\begin{align*}
			g\ast_f(h\ast_f a)=&f(g)\cdot \Big(g\ast(h\ast_f a)\Big)\\
			=&f(g)\cdot \Big(g\ast\big(f(h)\cdot (h\ast a)\big)\Big)\\
			=&f(g)\cdot \Big(g\ast f(h)\Big) \cdot \Big(g\ast (h\ast a)\Big).
		\end{align*}
		Since $f$ is a 1-cocycle, we have $$f(g)\cdot (g\ast f(h))=f(gh).$$ Therefore, the right-hand side becomes
		$$g\ast_f(h\ast_f a)=f(gh)\cdot (gh\ast a).$$
		This settles the second equality.
	\end{proof}
	We remark that the group elements  $g\in G$ can no longer be regarded as field automorphisms via the new action $\ast_f$, because the action is not multiplicative in general. Namely, we do not have $g\ast_f(a\cdot b)=(g\ast_f a)\cdot(g\ast_f b).$

	\subsection{The Galois module $U^i_L/U^{i+1}_L$}
	In this subsection, we consider the quotient $U^i_L/U^{i+1}_L$. We know that in the category of abelian groups, $U^i_L/U^{i+1}_L$ is isomorphic non-canonically to $\lambda$ when $i\geqslant1$. In this subsection, we discuss the Galois module structures rather than abelian groups. 
	
	We resume the previous notations. Let $G=\gal(L/K)$ and $i\geqslant1$. Let $s=1+u\cdot\pi_L^i$ be an element in $U^i_L$ but not in $U^{i+1}_L$. Then, for any $\sigma\in G$, we have
	\begin{align*}
		\sigma(s)=&1+\sigma(u)\cdot\sigma(\pi_L)^i\\
		=&1+\sigma(u)\cdot\dfrac{\sigma(\pi_L)^i}{\pi_L^i}\cdot \pi_L^i.
	\end{align*}
	Recall that the group isomorphism $U^{i}_L/U^{i+1}_L\xrightarrow{\sim}\lambda$ is given by the assignment 
	$s\mapsto\widehat{u}$, so we have $$\sigma(s)\mapsto\reallywidehat{\sigma(u)\cdot\dfrac{\sigma(\pi_L)^i}{\pi_L^i}}=\reallywidehat{\sigma(u)\cdot f^i_\pi(\sigma)},$$
	where $f_\pi$ is the $1$-cocycle considered in the previous sections. 
	
	From the above discussions, we see that in order for the map $U^i_L/U^{i+1}_L\to\lambda$ to be $G$-equivariant, it is necessary to give a new action of $G$ to $\lambda$ such that 
	$$\sigma\ast y=\sigma(y)\cdot\reallywidehat{f^i_\pi(\sigma)},$$ for any $y\in\lambda$.
	
	We see that the new group action on $\lambda$ is the Galois action twisted by $\reallywidehat{f^i_\pi}$, in the sense of the previous subsection, where $G\to G/G_0=\gal(\lambda/\kappa)$ is the natural quotient map. We denote this Galois module by $\lambda_i$. We also remark that when $L/K$ is unramified, then $\pi_L=\pi_K$, and $\sigma(\pi_L)/\pi_L$ is just $1$. Hence there is no need to consider twisting in the unramified case. We summarize our discussion as the following:
	\begin{Def}
		Let $\lambda$ and $G=\gal(L/K)$ be as above. We define the $G$-module $\lambda_i$, such that $\lambda_i=\lambda$ as abelian groups, and the action of $G$ on $\lambda_i$ is given by the formula 
		\begin{align*}
			\sigma\ast_i a=&\sigma(a)\cdot(\reallywidehat{f_\pi(\sigma)})^i\\
			=&\sigma(a)\cdot(\reallywidehat{\sigma(\pi_L)/\pi_L})^i,
		\end{align*}
		 for any $a\in\lambda$ and any $\sigma\in G$.  The $\ast_i$ is indeed a $G$-action because of Proposition 4.1.
 	\end{Def}
	
	We close this subsection by computing the group of invariants $(\lambda_i)^G$. 
	Let $a\in\lambda_i$ be an element which is invariant under the twisted action $\ast_i$ of $G$. Then, for any $\sigma\in G$, we have 
	$$a=\sigma\ast_i a=\sigma(a)\cdot \reallywidehat{f_\pi(\sigma)}^i.$$ 
	Since $0$ is trivially an invariant element, we may assume that $a\neq0$, hence $\sigma(a)\neq0$. Thus, the above equality can be written as
	$$\dfrac{a}{\sigma(a)}= \reallywidehat{f_\pi(\sigma)}^i.$$
	This implies that $f'=\reallywidehat{f_{\pi}}^i$ is a 1-coboundary on $\lambda^{\times}$. More precisely, this means the image of $(f_\pi)^i$ is trivial via the map $H^{1}(G,U_{L})\to H^{1}(G,\lambda^{\times})$. By the previous section, we know that the kernel of this map is exactly $H^{1}(G,U_{L})$ and is generated by $f_\pi^t$. Thus, we must have $t|i$, where $t$ is the tame ramification index.
	
	Hence, we draw the following fact:
	$$(\lambda_i)^G\neq0\implies t|i.$$
	Therefore, when $t$ does not divide $i$, we have $(\lambda_i)^G=0.$ Next, we consider the case where $t|i$. By the previous section, $f'$ is a 1-coboundary on $\lambda^{\times}$. Therefore, there exists $b\in\lambda^{\times}$ such that the following holds:
	 $$f'(\sigma)=\dfrac{\sigma(b)}{b}.$$ Hence, for any $a\in\lambda^{\times}$, we have 
	$$\sigma\ast_{i}a=\sigma(a)\cdot f'(\sigma)=\sigma(a)\cdot\dfrac{\sigma(b)}{b}.$$
	If $a$ is invariant under this action, then we have
	$$a=\dfrac{\sigma(a)\sigma(b)}{b},$$
	hence $$ab=\sigma(ab).$$
	This implies $ab\in(\lambda)^G=\kappa$, so $a\in b^{-1}\cdot\kappa,$ which is an additive subgroup of $\lambda$. We remark that this $b$ depends on the index $i$.
	
	\subsection{The first cohomology of $\lambda_i$}
	In this subsection, we consider the cohomology group $H^1(G,\lambda_i)$.
	The inflation-restriction sequence for $\lambda_i$ gives:
	$$0\to H^1(G/G_0, (\lambda_i)^{G_0})\to H^1(G, \lambda_i)\to H^1(G_0, \lambda_i)^{G/G_0}.$$
	To study these groups, we first look at the invariants $(\lambda_i)^{G_0}$. For $\sigma\in G_0$, we use the notations from the previous subsection and notice that $$\sigma\ast_{i}a=\sigma\ast_{f'}a=f'(\sigma)\cdot\sigma(a).$$
	We note that with the original Galois action, the inertia group $G_0$ acts trivially on the residue field $\lambda$, so $\sigma(a)=a$ for any $\sigma\in G_0$ and any $a\in\lambda$. Hence, the above equality becomes 
	$$\sigma\ast_{f'}a=f'(\sigma)\cdot a.$$
	Hence, if $a\in(\lambda_i)^{G_0}$, then we must have $f'(\sigma)\cdot a=a$. 
	Thus, the equality $$f'(\sigma)=\reallywidehat{f_{\pi}}^i(\sigma)=1,$$holds in $\lambda^{\times}$ for any $\sigma\in G_0$. By Proposition 1.4, this yields $t|i$. When this happens, we must have $(\lambda_i)^{G_0}=\lambda_i.$
	
	Therefore, we conclude:
	\[
	(\lambda_i)^{G_0}=
	\begin{cases*}
		0 & \text{if} $\quad t\nmid i,$ \\
		\lambda_i &\text{if} $\quad t|i$.
	\end{cases*}
	\]
	From this, it follows that 
	\[
	H^1(G/G_0, (\lambda_i)^{G_0})=
	\begin{cases*}
		0 & \text{if} $\quad t\nmid i$, \\
		H^1(G/G_0, \lambda_i) &\text{if} $\quad t|i$.
	\end{cases*}
	\]
	Our goal is to show that in the second case where $t|i$, the group $H^1(G/G_0, \lambda_i)$ is $0$. Since $\lambda_i$ is finite, its Herbrand quotient is 1. Thus, it suffices to show that the 0-th Tate cohomology group $\reallywidehat{H}^0(G/G_0, \lambda_i)$ is trivial.
	By definition, we have $$\reallywidehat{H}^0(G/G_0, \lambda_i)=\dfrac{(\lambda_i)^{G/G_0}}{N_{G/G_0}(\lambda_i)}.$$
	By the previous subsection, we know that $$(\lambda_i)^{G/G_0}=(\lambda_i)^G=b^{-1}\cdot\kappa,$$ where $b\in\lambda^{\times}$ is the element such that $$f'(\sigma)\coloneqq\reallywidehat{f_{\pi}}^i(\sigma)=\dfrac{\sigma(b)}{b}.$$ 
	 Therefore, it suffices to show that $$N_{G/G_0}(\lambda_i)=b^{-1}\cdot\kappa.$$ For simplicity, we denote by $\mathfrak{g}=G/G_0$. Note that $\mathfrak{g}=\gal(\lambda/\kappa)$, which is a cyclic group. Suppose $\tau$ is a generator of $\mathfrak{g}$, then we can compute $N_{\mathfrak{g}}(\lambda_i)$ as follows:
	
	For any element $a\in\lambda,$ we have 
	\begin{align*}
		N_{\mathfrak{g}}\ast_{i}a&=\sum\limits_{k=0}^{|\mathfrak{g}|-1}\tau^k\ast_{i}a\\
		&=\sum\limits_{k=0}^{|\mathfrak{g}|-1}f'(\tau^k)\cdot\tau^k(a)\\
		&=\sum\limits_{k=0}^{|\mathfrak{g}|-1}\dfrac{\tau^k(b)}{b}\cdot\tau^k(a)\\
		&=\sum\limits_{k=0}^{|\mathfrak{g}|-1}\dfrac{\tau^k(ba)}{b}\\
		&=\dfrac{\sum\limits_{k=0}^{|\mathfrak{g}|-1}\tau^k(ba)}{b}.
	\end{align*}
	In the above fraction, we notice that the numerator $\sum\limits_{k=0}^{|\mathfrak{g}|-1}\tau^k(ba)$ is the usual norm $\operatorname{N}_{\mathfrak{g}}(ba)$ with respect to the usual Galois action.
	When $a$ ranges through elements in $\lambda$, so does the product $ba$, since $b\neq0$ by definition. Thus, the numerator ranges through elements of $N_{\mathfrak{g}}(\lambda)$, with respect to the usual Galois action. By the normal basis theorem, the Galois module $\lambda$ is cohomologically trivial.  Thus, $H^1(\mathfrak{g}, \lambda)=0$, and it follows that $\reallywidehat{H}^0(\mathfrak{g}, \lambda)=0$ by the Herbrand quotient argument. This means $$N_{\mathfrak{g}}(\lambda)=(\lambda)^\mathfrak{g}=\kappa.$$
	 Therefore, we conclude that $$N_{\mathfrak{g}}(\lambda_i)=b^{-1}\kappa=(\lambda_i)^\mathfrak{g},$$ and this means $\reallywidehat{H}^0(G/G_0, \lambda_i)=0$, as desired. As a result, $H^1(\mathfrak{g}, (\lambda_i)^{G_0})$ is always 0, settling our claim. 
	 
	 Once we know that the group $H^1(G/G_0, (\lambda_i)^{G_0})$ is identically $0$, the inflation-restriction sequence 
	 $$0\to H^1(G/G_0, (\lambda_i)^{G_0})\to H^1(G, \lambda_i)\to H^1(G_0, \lambda_i)^{G/G_0}\to\cdots$$
	 now reads
	$$0\to H^1(G, \lambda_i)\to H^1(G_0, \lambda_i)^{G/G_0}\to\cdots$$
	In particular, the restriction map $H^1(G, \lambda_i)\xrightarrow{\text{res}} H^1(G_0, \lambda_i)$ is injective.
	
	We make some further considerations. For the quotient $G_0/G_1$ and the module $\lambda_i$, we have the inflation-restriction sequence
	$$0\to H^1(G_0/G_1, (\lambda_i)^{G_1})\to H^1(G_0, \lambda_i)\to H^1(G_1, \lambda_i)^{G/G_0}\to\cdots$$
	We note that the group $G_0/G_1$ has order $t$, which is the tame ramification index and is coprime to $p$. Thus, the abelian group $H^1(G_0/G_1, (\lambda_i)^{G_1})$ is eliminated by both $t$ and $p$, so it must be trivial. Therefore, the above sequence reads
	$$0\to H^1(G_0, \lambda_i)\to H^1(G_1, \lambda_i)^{G/G_0}\to\cdots$$
	In particular, the restriction $H^1(G_0, \lambda_i)\xrightarrow{\text{res}} H^1(G_1, \lambda_i)$ is injective. 
	
	Therefore, the restriction map $H^1(G, \lambda_i)\xrightarrow{\text{res}} H^1(G_1, \lambda_i)$ is injective as well, because it is equal to the composition of $H^1(G, \lambda_i)\xrightarrow{\text{res}} H^1(G_0, \lambda_i)$ and $H^1(G_0, \lambda_i)\xrightarrow{\text{res}} H^1(G_1, \lambda_i)$. 
	
	Next, we remark that the group $G_1$ acts trivially on $\lambda_i$, because for any $\sigma\in G_1$, $f_{\pi}(\sigma)=\dfrac{\sigma(\pi_L)}{\pi_L}\equiv1\mod\pi_L$.
	Therefore, $$\sigma\ast_{i}a=\reallywidehat{f_\pi(\sigma)}^i\cdot\sigma(a)=1\cdot\sigma(a)=\sigma(a)=a.$$
	The last equal sign is due to the fact that the subgroup $G_1\subseteq G_0$ has trivial Galois action on $\lambda$. Hence, $G_1$ indeed acts trivially on $\lambda_i$. As a result, $H^1(G_1, \lambda_i)=\Hom(G_1, \lambda)$. Namely, $H^1(G_1, \lambda_i)$ is the collection of group homomorphisms from $G_1$ to $\lambda=\lambda_i$ (as abelian groups). We summarize our discussion as the following:
	\begin{Prop}
		The restriction map $H^1(G, \lambda_i)\xrightarrow{\text{res}} H^1(G_1, \lambda_i)$ is injective, and we have $H^1(G_1, \lambda_i)=\Hom(G_1, \lambda)$. 
	\end{Prop}

	\section{Cohomology of higher unit groups}
	For simplicity, in this section we shall assume that the local fields $L/K$ are both local number fields, i.e., they are finite extensions of the field of $p$-adic numbers $\mathbb{Q}_p$. In particular, we have $\charar(L)=0.$
	\subsection{The assumption $t>i$}

	From Section 4, we have the following short exact sequence of $G$-modules:
	$$1\to U^{i+1}_L\to U^i_L\to \lambda_i\to 0.$$
	This exact sequence induces a long exact sequence 
	$$1\to(U^{i+1}_L)^{G}\to(U^i_L)^G\to(\lambda_i)^G\to H^1(G,U^{i+1}_L)\to H^1(G,U^{i}_L)\to H^1(G,\lambda_i)\to\cdots$$ 
	When $t$ does not divide $i$, we have $(\lambda_i)^G=0$. 
	In particular, this is true when $t>i$.
	Hence, in this case, the above long exact sequence yields:
	$$0\to H^1(G,U^{i+1}_L)\to H^1(G,U^{i}_L)\to H^1(G,\lambda_i)\to \cdots$$
	From this, we draw: 
	$$H^1(G,U^{i+1}_L)=\Ker(H^1(G,U^{i}_L)\to H^1(G,\lambda_i)).$$
	At the end of the previous section, we showed that the restriction map $H^1(G,\lambda_i)\xrightarrow{\text{res}} H^1(G_1,\lambda_i)$ is injective. Therefore, we may consider the composition 
	$$H^1(G,U^{i}_L)\to H^1(G,\lambda_i)\xrightarrow{\text{res}} H^1(G_1,\lambda_i),$$
	without changing the kernel: 
	\begin{align*}
		H^1(G,U^{i+1}_L)&=\Ker\Big(H^1(G,U^{i}_L)\to H^1(G,\lambda_i)\Big)\\
		&=\Ker\Big(H^1(G,U^{i}_L)\to H^1(G_1,\lambda_i)\Big).
	\end{align*}
	In particular, we have 
	\begin{Prop}
		Suppose $t$ does not divide $i$, where $t$ is the tame ramification index. Then, $H^1(G,U^{i+1}_L)$ is a subgroup of $H^1(G,U^{i}_L)$.
	\end{Prop}
	When $i=t, 2t, 3t, \cdots$. It is not easy to determine the structure of $H^1(G,U^{i+1}_L)$ and $H^1(G,U^{i}_L)$. Therefore, it is relatively easy to assume that $i$ is small. In the next two subsections, we will study the cases where $i=1,2$ to determine $U_L^2$ and $U_L^3$, under the assumption $t>i$.
	 
	 \subsection{First cohomology of $U^{2}_L$}
	 For $i=1$ and $t>i=1$, we draw the following from the previous discussions:
	 $$H^1(G,U^{2}_L)=\Ker\big(H^1(G,U^{1}_L)\to \Hom(G_1,\lambda_1)\big).$$
	 
	 In Section 3, we showed that the group $H^1(G,U^{1}_L)$ is a cyclic group of order $w$, 
	 with a generator being $(f_{\pi})^t$. 
	 We denote by $\beta$ the image of of $(f_{\pi})^t$ in $\Hom(G_1,\lambda_1)$. 
	 So $\beta$ is the assignment
	 $$\beta: \sigma\mapsto \reallywidehat{u_{\sigma,t}},$$
	where $\sigma\in G_1$, $ \reallywidehat{u_{\sigma,t}}$ is the image of $u_{\sigma,t}$ in the residue field $\lambda$, and $u_{\sigma,t}$ 
	is the element such that 
	$$\dfrac{\sigma(\pi_L^t)}{\pi_L^t}=1+u_{\sigma,t}\cdot\pi_L.$$
	
	We first remark that $\beta$ factors through the quotient group $G_1/G_2$, 
	because for elements $\sigma\in G_2$, 
	we have $\dfrac{\sigma(\pi_L)}{\pi_L}\in U^2_L$. As a result,  $$\dfrac{\sigma(\pi_L^t)}{\pi_L^t}=\Big(\dfrac{\sigma(\pi_L)}{\pi_L}\Big)^t\in U^2_L.$$
	Namely, $\beta(\sigma)=0$ for $\sigma\in G_2.$ Hence, we may regard $\beta$ as an element in the subgroup $\Hom(G_1/G_2,\lambda_1)\subseteq\Hom(G_1,\lambda_1)$. We also notice that $\beta$ is the image of the generator of $H^1(G,U^{1}_L)$, so the entire image of the map $H^1(G,U^{1}_L)\to \Hom(G_1,\lambda_1)$ is contained in the subset $\Hom(G_1/G_2,\lambda_1)$, i.e.,
	$$\im\Big(H^1(G,U^{1}_L)\to \Hom(G_1,\lambda_1)\Big)\subseteq\Hom(G_1/G_2,\lambda_1).$$
	
	We first consider the trivial case where $G_1/G_2=1$. Then, the group $\Hom(G_1/G_2,\lambda_1)$ is trivial. Thus, the above inclusion implies:
	$$\im\Big(H^1(G,U^{1}_L)\to \Hom(G_1,\lambda_1)\Big)=0,$$ and hence 
	$$\Ker\big(H^1(G,U^{1}_L)\to \Hom(G_1,\lambda_1)\big)=H^1(G,U^{1}_L).$$
	Therefore, we draw $H^1(G,U^{2}_L)=H^1(G,U^{1}_L).$

	Next, we assume that $G_1/G_2\neq1$. 
	
	Suppose, for $\sigma\in G_1$,
	\begin{equation}\label{u_sigma}
		\dfrac{\sigma(\pi_L)}{\pi_L}=1+u_{\sigma}\cdot\pi_L, 
	\end{equation}
	then we see that 
	$$\dfrac{\sigma(\pi_L^t)}{\pi_L^t}=\Big(\dfrac{\sigma(\pi_L)}{\pi_L}\Big)^t=(1+u_{\sigma}\cdot\pi_L)^t=1+t\cdot u_\sigma\cdot\pi_L+\cdots,$$
	where the omitted terms are all divisible by $(\pi_L)^2$.
	As a result, we have
	$$u_{\sigma,t}\equiv t\cdot u_{\sigma}\mod \pi_L,$$
	i.e.,
	$\reallywidehat{u_{\sigma,t}}=t\cdot\reallywidehat{u_{\sigma}}$.
	We remark that $t$ is coprime to $p=\charar(\lambda)$, so $t$ is invertible in $\lambda$.
	
	A similar calculation shows that $$\beta^p(\sigma)=p\cdot\reallywidehat{u_{\sigma,t}}=0,$$ since $p=\charar(\lambda)$. This implies that $(f_{\pi})^{pt}$ is in the kernel  of the map: 
	$$(f_{\pi})^{pt}\in H^1(G,U^{2}_L)=\Ker\Big(H^1(G,U^{1}_L)\to\Hom(G_1,\lambda_1)\Big).$$
	
	Lastly, we show that $f_{\pi}^{pt}$ generates the kernel $H^1(G,U^{2}_L)$. Notice that $H^1(G,U^{1}_L)$ is a cyclic group, and $\langle f_{\pi}^{pt}\rangle\subseteq H^1(G,U^{1}_L)$ is the unique subgroup of index $p$. Thus, to show that the kernel is $\langle f_{\pi}^{pt}\rangle$, it suffices to show that the generator $f_\pi^t$ is not in the kernel.
	
	Recall that $\beta\in\Hom(G_1/G_2,\lambda_1)\subseteq\Hom(G_1,\lambda_1)$. Consider $\gamma\in\Hom(G_1/G_2,\lambda_1)$, which is the composition $$\gamma:G_1/G_2\xrightarrow{\theta_{1}} U^1_L/U^2_L \xrightarrow{\rho_1}\lambda=\lambda_1.$$ 
	  $\gamma$ is a non-trivial homomorphism because both $\theta_{1}$ and $\rho_1$ are injective maps, hence $\gamma$ is injective. We see that $\gamma$ is the assignment
	  $$\gamma:\sigma\mapsto\reallywidehat{u_\sigma},$$
	  where $u_\sigma$ is the element in the formula (\ref{u_sigma}). Therefore, we find the following relation:
	  $$\beta(\sigma)=t\cdot\gamma(\sigma),$$
	  for any $\sigma\in G_1.$ 
	  Since $t$ in invertible in $\lambda$, $\gamma$ being injective implies that $\beta$ is also injective. Thus $\beta\neq0,$ as desired.

	Let us summarize the discussions in this subsection:
	\begin{Prop}\label{propU2}
		Suppose that $t>1$, where $t$ the tame ramification index. Then, the group $H^1(G,U^{2}_L)$ is a subgroup of $H^1(G,U^{1}_L)$. Moreover, the following hold:
		\begin{itemize}
			\item If $G_1/G_2=1$, then $H^1(G,U^{2}_L)=H^1(G,U^{1}_L)$.
			\item If $G_1/G_2\neq1$, then $H^1(G,U^{2}_L)$ is generated by $(f_\pi)^{pt}$. In particular, it is the unique subgroup of $H^1(G,U^{1}_L)$ with index $p$ and its order is $w/p$.
		\end{itemize}  
	\end{Prop}
	
	Next, we remove the working assumption $t>1$, so $t=1$.

	 \subsection{First cohomology of $U^{3}_L$}
	We now consider $i=2$ and $t>i=2$. Again, we consider the sequence of $G$-modules
	$$1\to U^{3}_L\to U^2_L\to \lambda_2\to 0,$$
	which yields 
	 $$H^1(G,U^{3}_L)=\Ker\Big(H^1(G,U^{2}_L)\to \Hom(G_1,\lambda_2)\Big).$$
	 Based on Proposition \ref{propU2}, we have to discuss whether $G_1/G_2=1$ or not.
	 \subsubsection{$G_1/G_2=1$}
	 If $G_1/G_2=1$, then we have $G_1=G_2$. 
	 As a result, we have
	  $$H^1(G_1,\lambda_2)=H^1(G_2,\lambda_2)=\Hom(G_2,\lambda_2),$$ and we shall study the map
	  $$H^1(G,U^{2}_L)\to \Hom(G_2,\lambda_2)$$
	  By Proposition \ref{propU2}, we have $H^1(G,U^{2}_L)=H^1(G,U^{1}_L)$, in particular, $H^1(G,U^{2}_L)$ is generated by $f_\pi^t.$
	  As before, we denote by $\beta$ the image of $f_\pi^t$ in the group $\Hom(G_2,\lambda_2)$. Thus, $\beta$ is the assignment 
	  $$\beta:\sigma\mapsto\reallywidehat{s_{\sigma,t}},$$
	  where $\sigma\in G_2$ and $s_{\sigma,t}$ is the element such that 
	 $$\dfrac{\sigma(\pi_L^t)}{\pi_L^t}=1+s_{\sigma,t}\cdot(\pi_L)^2.$$
	Like before, this map $\beta$ factors through the quotient $G_2/G_3$. Hence,
	$$\im\Big(H^1(G,U^{2}_L)\to \Hom(G_2,\lambda_2)\Big)\subseteq\Hom(G_2/G_3,\lambda_2)\subseteq\Hom(G_2,\lambda_2).$$
	Again, if $G_2/G_3=1$, then the group $\Hom(G_2/G_3,\lambda_2)$ is trivial. We draw:
	$$\Ker\big(H^1(G,U^{2}_L)\to \Hom(G_2,\lambda_2)\big)=H^1(G,U^{2}_L).$$
	 This means $H^1(G,U^{3}_L)=H^1(G,U^{2}_L)=H^1(G,U^{1}_L)$.
	 Moreover, we can prove the following by induction:
	 \begin{Prop}
	 	Suppose $G_1=G_2=\cdots=G_j$ for some $j\geqslant2$ and $t>j$. Then the following equality holds:
	 	$$H^1(G,U^{1}_L)=H^1(G,U^{2}_L)=\cdots H^1(G,U^{j}_L).$$
	 \end{Prop}
	
	Next, we consider the non-trivial situation where $G_2/G_3\neq1$. As before, we consider the composition
	 $$\gamma:G_2/G_3\xrightarrow{\theta_{2}} U^2_L/U^3_L \xrightarrow{\rho_2}\lambda=\lambda_2.$$ 
	 This $\gamma$ is an injective map, hence is non-zero.
	 We compare $\gamma$ with $\beta$ in the same way as the previous subsection, and the same relation
	 $\beta(\sigma)=t\cdot\gamma(\sigma)$ holds.
	The same argument then yields 
	$$H^1(G,U^{3}_L)=\langle f_\pi^{pt}\rangle.$$
	\subsubsection{$G_1/G_2\neq1$}
	When $G_1/G_2$ is non-trivial, by Proposition \ref{propU2} we know that $H^1(G,U^{2}_L)$ is generated by $f_\pi^{pt}$. Let $\beta$ be the image of $f_\pi^{pt}$ in $\Hom(G_1,\lambda_2)$. For $\sigma\in G_1$, we have 
	$$f_\pi^{pt}(\sigma)=\Big(\dfrac{\sigma(\pi_L)}{\pi_L}\Big)^{pt}.$$
	Write as before $\dfrac{\sigma(\pi_L)}{\pi_L}=1+u_\sigma\cdot\pi_L,$
	then 
	\begin{align}\label{pt-th power}
		\Big(\dfrac{\sigma(\pi_L)}{\pi_L}\Big)^{pt}&=(1+u_\sigma\cdot\pi_L)^{pt} \nonumber \\
		&=1+pt\cdot u_\sigma\cdot\pi_L+{pt\choose 2}\cdot (u_\sigma)^2\cdot(\pi_L)^2+\cdots,
	\end{align}
	where the omitted terms are divisible by $(\pi_L)^3$.
	We denote by $e_L$ the ramification index of $L/\mathbb{Q}_p$ and by $e_K$ the ramification index of $K/\mathbb{Q}_p$, so $e_L=e\cdot e_K$. Then we have $$p=s\cdot(\pi_L)^{e_L}$$ for some unit $s$ in $L$. Note that $G_1/G_2\neq1$ implies $G_1\neq1$, hence we have $$|G_1|=w\geqslant p.$$
	In particular, $e_L\geqslant p\geqslant2,$ so $e_L+1\geqslant3$.
	Thus, formula (\ref{pt-th power}) becomes 
	\begin{align*}
			\Big(\dfrac{\sigma(\pi_L)}{\pi_L}\Big)^{pt}&=1+pt\cdot u_\sigma\cdot\pi_L+{pt\choose 2}\cdot (u_\sigma)^2\cdot(\pi_L)^2+\cdots \\
			&=1+ts\cdot u_\sigma\cdot(\pi_L)^{e_L+1}+{pt\choose 2}\cdot (u_\sigma)^2\cdot(\pi_L)^2+\cdots\\
			&=1+{pt\choose 2}\cdot (u_\sigma)^2\cdot(\pi_L)^2+\cdots
	\end{align*}
	We notice that $p|{pt\choose 2}$, hence the above considerations apply to the remaining term ${pt\choose 2}\cdot (u_\sigma)^2\cdot(\pi_L)^2$. It follows that $(\pi_L)^3$ divides ${pt\choose 2}\cdot (u_\sigma)^2\cdot(\pi_L)^2$. Thus, we draw $\Big(\dfrac{\sigma(\pi_L)}{\pi_L}\Big)^{pt}\in U^3_L$.
	As a result, $$\Big(\dfrac{\sigma(\pi_L)}{\pi_L}\Big)^{pt}\mapsto0\in\lambda_2\cong U^2_L/U^3_L.$$
	This means $\beta\in\Hom(G_1,\lambda_2)$ is the trivial homomorphism, and therefore $H^1(G,U^{2}_L)\to \Hom(G_1,\lambda_2)$ is trivial. Hence, 
	$$H^1(G,U^{3}_L)=\Ker\Big(H^1(G,U^{2}_L)\to \Hom(G_1,\lambda_2)\Big)=H^1(G,U^{2}_L).$$
	We conclude:
	\begin{Prop}
		Suppose $t>2$ and $G_1/G_2\neq1$. Then, we have 
		$$H^1(G,U^{3}_L)=H^1(G,U^{2}_L)=\langle f_\pi^{pt}\rangle.$$
	\end{Prop}
	If we combine the above argument with induction, then the following holds:
	\begin{Prop}
		Suppose $G_1/G_2\neq1$. Then, for any integer $j$ such that $j\leqslant t$ and $2\leqslant j\leqslant e_L+1$, we have:
		$$H^1(G,U^{j}_L)=H^1(G,U^{2}_L)=\langle f_\pi^{pt}\rangle.$$
	\end{Prop}
	\subsection{Discussions on general situations}
	In this subsection, we discuss the general situation for $U_L^{i+1}\subseteq U_L^{i}$, where the assumption $t>i$ is removed. In Section 5.1, we already showed that $H^1(G,U^{i+1}_L)$ is a subgroup of $H^1(G,U^{i}_L)$ if $t$ does not divide $i$. In what follows, we assume $t|i$.
		
	At the beginning we had the sequence
	$$1\to U^{i+1}_L\to U^i_L\to \lambda_i\to 0.$$
	which yields a long exact sequence:
	\begin{multline}\label{longseq}
		1\to(U^{i+1}_L)^{G}\to(U^i_L)^G\to(\lambda_i)^G\to \\
		H^1(G,U^{i+1}_L)\to H^1(G,U^{i}_L)\to H^1(G,\lambda_i)\to\cdots
	\end{multline}
	If $t$ divides $i$, then $(\lambda_i)^G\neq0$, and we have found that $(\lambda_i)^G=b^{-1}\cdot\kappa$, with $b$ depending on $i$. Thus $(\lambda_i)^G$ is isomorphic to $\kappa$ as abelian groups.
	\begin{Prop}
		If the ramification index $e$ divides $i$, then $H^1(G,U^{i+1}_L)$ is a subgroup of $H^1(G,U^{i}_L)$.
	\end{Prop}
	\begin{proof}
		We have figured out that 
		$$(U^i_L)^G=U_K^{\lceil\frac{i}{e}\rceil}.$$
		When $e|i$, we have 
		${\lceil\dfrac{i+1}{e}\rceil}=\dfrac{i}{e}+1=\lceil\dfrac{i}{e}\rceil+1.$ 
		If we denote the quotient $\dfrac{i}{e}$ by $j$, then we have 
		$$(U^i_L)^G=U_K^{j}, \ \text{and}\ (U^{i+1}_L)^G=U_K^{j+1}.$$
		As a result, the long exact sequence (\ref{longseq}) becomes:
		$$1\to U_K^{j+1}\to U_K^{j}\to(\lambda_i)^G\to 
		H^1(G,U^{i+1}_L)\to H^1(G,U^{i}_L)\to H^1(G,\lambda_i)\to\cdots.$$
		The image of the map $U_K^{j}\to(\lambda_i)^G$ is isomorphic to $U_K^{j}/U_K^{j+1}\cong\kappa$. Therefore, this image has the same cardinality with $(\lambda_i)^G$. As a result, the map in the above sequence $U_K^{j}\to(\lambda_i)^G$ is surjective. The surjectivity of $U_K^{j}\to(\lambda_i)^G$ breaks the long sequence into two parts, and the second part reads:
		$$0\to H^1(G,U^{i+1}_L)\to H^1(G,U^{i}_L)\to H^1(G,\lambda_i)\to\cdots.$$
		This settles the proof.
	\end{proof}
	
	Next, we assume that $e$ does not divide $i$. We remark that this (along with $t|i$) implies $w\neq1$, which is our assumption from the beginning. Then, we have 
	$${\lceil\dfrac{i+1}{e}\rceil}=\lceil\dfrac{i}{e}\rceil.$$	
	With the same notation as above, we then have 
	$$(U^i_L)^G=(U^{i+1}_L)^G=U_K^{j}.$$
	As a result, the sequence (\ref{longseq}) now becomes 
	$$1\to U_K^{j}\to U_K^{j}\to(\lambda_i)^G\to 
	H^1(G,U^{i+1}_L)\to H^1(G,U^{i}_L)\to H^1(G,\lambda_i)\to\cdots.$$
	The map $U_K^{j}\to U_K^{j}$ is induced from the inclusion map $U_L^{i+1}\to U_L^{i}$ via taking invariants, hence $U_K^{j}\to U_K^{j}$ is the identity map in the above long exact sequence.
	Therefore, we draw:
	$$0\to(\lambda_i)^G\to 
	H^1(G,U^{i+1}_L)\to H^1(G,U^{i}_L)\to H^1(G,\lambda_i)\to\cdots.$$
	In particular, if we denote by $K_i$ the kernel of the map $H^1(G,U^{i}_L)\to H^1(G,\lambda_i)$, then the following holds:
	\begin{Prop}
		Suppose $t$ divides $i$ but $e$ does not divide $i$, then there is a short exact sequence of abelian groups:
		$$0\to(\lambda_i)^G\to H^1(G,U^{i+1}_L)\to K_i\to0.$$
	\end{Prop}
	
\end{document}